\documentclass[12pt]{article}
\input epsf.tex

\usepackage{graphicx}
\usepackage{amsmath,amsthm,amsfonts,amscd,amssymb,comment,eucal,latexsym,mathrsfs}
\usepackage{stmaryrd}
\usepackage[all]{xy}

\usepackage{epsfig}

\usepackage[all]{xy}
\xyoption{poly}
\usepackage{fancyhdr}
\usepackage{wrapfig}
\usepackage{epsfig}

\theoremstyle{plain}
\newtheorem{thm}{Theorem}[section]
\newtheorem{prop}[thm]{Proposition}
\newtheorem{lem}[thm]{Lemma}
\newtheorem{cor}[thm]{Corollary}

\theoremstyle{definition}
\newtheorem{defn}{Definition}
\theoremstyle{remark}
\newtheorem{remark}{Remark}

\advance \topmargin by -\headheight
\advance \topmargin by -\headsep
\evensidemargin \oddsidemargin
  \def\C{{\mathbb{C}}}  \def\E{{\mathbb{E}}} \def\F{{\mathbb{F}}}        \def\N{{\mathbb{N}}}            \def\Z{{\mathbb{Z}}}

\def\bfa{{\bf{a}}} \def\bfb{{\bf{b}}}                        

 \def\cB{{\mathcal{B}}} \def\cC{{\mathcal{C}}}   \def\cF{{\mathcal{F}}} \def\cG{{\mathcal{G}}}         \def\cP{{\mathcal{P}}} \def\cQ{{\mathcal{Q}}} \def\cR{{\mathcal{R}}} \def\cS{{\mathcal{S}}}      \def\cY{{\mathcal{Y}}} \def\cZ{{\mathcal{Z}}}

     \def\tcF{{\tilde{\mathcal{F}}}} \def\tcG{{\tilde{\mathcal{G}}}}

   \def\sD{{\mathscr{D}}}           \def\sO{{\mathscr{O}}} \def\sP{{\mathscr{P}}}          

                       \def\tX{{\tilde{X}}}  

     \def\tf{{\tilde{f}}}                    

     \def\tcF{{\tilde{\cF}}} \def\tcG{{\tilde{\cG}}}

   \def\tmu{{\widetilde{\mu}}}             

\newcommand{\Ga}{\Gamma}
\newcommand{\ga}{\gamma}

\newcommand{\eps}{\epsilon}

\newcommand\Aut{\operatorname{Aut}}

\newcommand\Cay{\operatorname{Cay}}

\newcommand\Ca{\operatorname{{\bf Cantor}}}
\newcommand\Cantor{\textrm{{\bf Cantor}}}

\newcommand\Factor{{\operatorname{Factor}}}
\newcommand\Fix{{\operatorname{Fix}}}

\newcommand\Ker{\operatorname{Ker}}

\newcommand\Prob{\operatorname{Prob}}

\def\cc{{\curvearrowright}}

\begin{document}
\title{Zero entropy is generic}
\author{Lewis Bowen\footnote{supported in part by NSF grant DMS-1500389, NSF CAREER Award DMS-0954606}}
\maketitle

\begin{abstract}
Dan Rudolph showed that for an amenable group $\Ga$, the generic measure-preserving action of $\Ga$ on a Lebesgue space has zero entropy. Here this is extended to nonamenable groups. In fact, the proof shows that every action is a factor of a zero entropy action! This uses the strange phenomena that in the presence of nonamenability, entropy can increase under a factor map. The proof uses Seward's recent generalization of Sinai's Factor Theorem, the Gaboriau-Lyons result and my theorem that for every nonabelian free group, all Bernoulli shifts factor onto each other. 
\end{abstract}


\noindent
\tableofcontents

\section{Introduction}

Entropy theory in dynamics has recently been extended from actions of the integers (and more generally, amenable groups) to actions of sofic groups \cite{bowen-jams-2010} and arbitrary countable groups \cite{seward-kreiger-1, seward-kreiger-2, alpeev-seward}. Here we begin to investigate generic properties of measure-preserving actions of countable groups with an eye towards understanding their entropy theory.

Our starting point is a result due to Rokhlin \cite{rokhlin-1959}: the generic automorphism $T \in \Aut(X,\mu)$ has zero entropy. To be precise, $(X,\mu)$ denotes a Lebesgue probability space and $\Aut(X,\mu)$ is the group of measure-preserving automorphisms $\phi:X \to X$ in which automorphisms that agree almost everywhere are identified. This group has a natural Polish topology: a sequence $\{T_i\} \subset\Aut(X,\mu)$ converges to $T$ if for every measurable subset $A \subset X$, $\mu(T_i A \vartriangle TA)  \to 0$ as $i\to\infty$. The claim is that the subset of all transformations $T \in \Aut(X,\mu)$ that have zero entropy contains a dense $G_\delta$ subset so that it is residual in the sense of Baire category.

In order to consider the analogous question for general countable groups, we first need a notion of entropy. So suppose we have a countable group $\Ga$ and a probability-measure-preserving action $\Ga \cc (X,\mu)$. Assuming the action is ergodic, its {\bf Rokhlin entropy}, denoted $h^{Rok}_\Ga(X,\mu)$, is the  infimum of $H_\mu(\cP)$ over all generating partitions $\cP$. Recall that a partition $\cP$ of $X$ is {\bf generating} if the smallest $\Ga$-invariant sigma-algebra containing it is the full Borel sigma-algebra (modulo null sets) and the Shannon entropy is defined by
$$H_\mu(\cP) : = -\sum_{P\in \cP} \mu(P)\log\mu(P).$$
Rokhlin entropy agrees with Kolmogorov-Sinai entropy for essentially free actions whenever $\Ga$ is amenable \cite{seward-tucker-drob} and Rokhlin entropy upper-bounds sofic entropy when $\Ga$ is sofic (this is immediate from the definition in \cite{bowen-jams-2010}). 

We also need a space of actions. This can be handled in two different ways. We consider the space $A(\Ga,X,\mu)$ of all homomorphisms $\alpha:\Ga \to \Aut(X,\mu)$ equipped with the topology of pointwise convergence (see \cite{Kechris-global-aspects} for details). Alternatively, let $\Ca$ denote the usual middle thirds Cantor set and let $\Ga$ act on $\Ca^\Ga$ by $(fx)(g)=x(f^{-1}g)$ (where $x\in \Ca^\Ga$ is represented as a function $x:\Ga \to \Ca$). This action is by homeomorphisms when we equip $\Ca^\Ga$ with the product topology. We let $\Prob_\Ga(\Ca^\Ga)$ denote the space of all $\Ga$-invariant Borel probability measures on $\Ca^\Ga$ with respect to the weak* topology. A fundamental result of Glasner-King \cite{GK98} together with the weak Rokhlin property \cite{glasner2006every} implies that if $\cP$ is any property of actions that is invariant under measure-conjugacy then the set of all actions $\alpha \in A(\Ga,X,\mu)$ that have $\cP$ is a residual set if and only if the set of all measures $\mu \in \Prob_\Ga(\Ca^\Ga)$ such that $\Ga \cc (\Ca^\Ga,\mu)$ has $\cP$ is a residual set\footnote{More precisely, Glasner and King proved this result with the unit interval in place of the Cantor set. However, in \cite{bowen-hartman-tamuz-1} it was shown to hold for any perfect Polish space in place of the unit interval.}. Therefore, we can choose to study either $A(\Ga,X,\mu)$ or $\Prob_\Ga(\Ca^\Ga)$, whichever one is most convenient for the problem at hand. For most of the paper, we use $\Prob_\Ga(\Ca^\Ga)$ and state the results in terms of $A(\Ga,X,\mu)$.

The first result of this paper is:
\begin{thm}\label{thm:zero}
For any countably infinite group $\Ga$, the subset of actions $\bfa \in A(\Ga,X,\mu)$ with zero Rokhlin entropy is residual in the sense of Baire category. 
\end{thm}

As mentioned above, because Rokhlin entropy is an upper bound for sofic entropy, this implies that the generic action $\bfa \in A(\Ga,X,\mu)$ has nonpositive sofic entropy with respect to all sofic approximations of $\Ga$. 


The main difficulty in proving Theorem \ref{thm:zero} is showing that the subset of actions with zero entropy is dense. If $\Ga$ is amenable then the argument is due to Rudolph (see the Subclaim after Claim 19 in \cite{foreman-weiss}). It is essentially a consequence of the Rokhlin Lemma which implies if an action $\bfa \in A(\Ga,X,\mu)$ is essentially free then its measure-conjugacy class is dense in $A(\Ga,X,\mu)$. If $\Ga$ is nonamenable, then this no longer holds: for example if $\bfa$ is strongly ergodic (e.g. if it is a Bernoulli shift) then the closure of its measure-conjugacy class does not contain any nonergodic actions. 

Assuming $\Ga$ is nonamenable, we take advantage of the fact that entropy can {\em increase} under a factor map. The first example of this phenomenon is due to Ornstein and Weiss \cite{OW87}; they showed that the 2-shift over the rank 2 free group factors onto the 4-shift. This was generalized in several ways: Ball proved that if $\Ga$ is any nonamenable group then there exists some probability space $(K,\kappa)$ with $|K|<\infty$ such that the Bernoulli shift $\Ga \cc (K,\kappa)^\Ga$ factors onto all Bernoulli shifts over $\Ga$ \cite{ball-factors1}. I proved that if $\Ga$ contains a nonabelian free group then in fact all Bernoulli shifts over $\Ga$ factor onto each other \cite{bowen-ornstein-2011}. It is still unknown whether this conclusion holds for all nonamenable $\Ga$. Lastly, Seward proved there is some number $r(\Ga)<\infty$ depending only on $\Ga$ such that if $\Ga \cc (X,\mu)$ is an arbitrary measure-preserving action then there exists another action $\Ga \cc (\tX,\tmu)$ with Rokhlin entropy $\le r(\Ga)$ that factors onto it \cite{seward-small-action}. In other words, every action has an extension with bounded Rokhlin entropy. Our next result shows we can take $r(\Ga)=0$:

\begin{thm}\label{thm:extension}
If $\Ga$ is nonamenable and $\Ga \cc (X,\mu)$ is essentially free, ergodic and probability-measure-preserving, then there exists an action $\Ga \cc (\tX,\tmu)$ with zero Rokhlin entropy that extends $\Ga \cc (X,\mu)$.
\end{thm}

Here is a quick sketch of the proof: using the ideas of Gaboriau-Lyons \cite{gaboriau-lyons} and the fact that, for free groups, all Bernoulli shifts factor onto each other \cite{bowen-ornstein-2011}, it is shown that there exists an inverse limit of factors of Bernoulli shifts which (a) has zero Rokhlin entropy and (b) factors onto all Bernoulli shifts. (By contrast, if $\Ga=\Z$ consequences of Ornstein theory imply that inverse limits and factors of Bernoulli shifts are Bernoulli \cite{ornstein-1970c, MR0447525}). Without loss of generality, we may assume $\Ga \cc (X,\mu)$ has positive Rokhlin entropy. Using Seward's recent spectacular generalization of Sinai's Factor Theorem \cite{seward-sinai} the extension $\Ga \cc (\tX,\tmu)$ is constructed as a relatively independent joining of $\Ga \cc (X,\mu)$ and this inverse limit over a common Bernoulli factor. A standard argument shows that since $\Ga \cc (X,\mu)$ is a factor of a zero entropy action, it is also a limit of zero entropy actions (see Lemma \ref{lem:2}), proving that zero entropy actions are dense . 

\subsection{Strengthenings of zero entropy}

Theorem \ref{thm:extension} highlights the fact that, if $\Ga$ is nonamenable, zero entropy actions can have positive entropy factors. So we consider the following stronger notions of zero entropy for an action $\bfa=\Ga \cc (X,\mu)$:
\begin{enumerate}
\item $\bfa$ has {\bf completely zero entropy} (this means every essentially free factor of $\bfa$ has zero Rokhlin entropy);
\item $\bfa$ is disjoint from all Bernoulli shifts over $\Ga$;
\item $\bfa$ is disjoint from all R-CPE (completely positive Rokhlin entropy) actions of $\Ga$;
\item every factor of every self-joining (including infinite self-joinings) of $\bfa$ has zero Rokhlin entropy;
\item $\bfa$ has zero naive entropy (naive entropy is defined in \S \ref{sec:naive}). 
\end{enumerate}
If $\Ga$ is amenable then all five notions agree with zero entropy. In \S \ref{sec:strengthening} it is shown that (for any group $\Ga$) $1 \Leftarrow 2$ and  $3 \Leftarrow 4 \Leftarrow 5$. Moreover, if $\Ga$ is sofic then $2 \Leftarrow 3$. It is an open problem whether all of these properties are equivalent.

To state the next result, recall that a group $\Ga$ has property MD if the measure-conjugacy class of direct product of the action of $\Ga$ on its profinite completion by left-translations with the trivial action on the unit interval is dense in the space $A(\Ga,X,\mu)$ of actions \cite{kechris-2012}. For example, free groups, surface groups and fundamental groups of hyperbolic 3-manifolds have MD (Theorem \ref{thm:md10} below). The final result shows that, for some groups, zero naive entropy is generic:
\begin{thm}\label{thm:md2}
Suppose $\Ga$ either has property MD  or has the form $\Ga = G \times H$ where $H$ is an infinite amenable residually finite group. Then the subset of all actions $\bfa \in A(\Ga,X,\mu)$ with zero naive entropy is residual in the sense of Baire category.
\end{thm}
It is an open problem whether this conclusion holds for every group $\Ga$. Indeed, it is unknown whether every group $\Ga$ admits an essentially free action with zero naive entropy.

The notion of weak containment of actions was introduced by Kechris \cite{kechris-2012} as an analog to weak containment of unitary representations. For a given action $\bfa$ it is an open problem whether the generic action that is  weakly equivalent to $\bfa$  has zero Rokhlin entropy. However, if $\bfa$ is a Bernoulli shift then we show this is the case in the last section \S \ref{sec:weak}.

\subsection{Organization}
\S \ref{sec:prelim} introduces notation and recalls important terminology. \S \ref{sec:rokhlin} reviews Rokhlin entropy and proves that zero Rokhlin entropy is a $G_\delta$ condition for essentially free, ergodic actions. \S \ref{sec:inverse-limit} constructs an inverse limit of factors of Bernoulli shifts that has zero Rokhlin entropy and factors onto all Bernoulli shifts. \S \ref{sec:extension} proves Theorem \ref{thm:extension}. \S \ref{sec:zero} proves Theorem \ref{thm:zero}. \S \ref{sec:naive} introduces naive entropy. \S \ref{sec:strengthening} introduces five strengthenings of zero entropy. \S \ref{sec:zne} proves Theorem \ref{thm:md2}. The last section \S \ref{sec:weak} formulates the open problem: for a given weak equivalence classes of actions, is zero entropy generic? 

{\bf Acknowledgements}. I am deeply grateful to Robin Tucker-Drob and Brandon Seward. Many of the ideas presented here were obtained during conversations with each of them, spanning over a year. Also thanks to Miklos Abert for suggesting the problem of determining whether zero entropy is generic in each weak equivalence class. And thanks to Pierre-Antoine Guih\'eneuf for the reference \cite{rokhlin-1959} and Benjy Weiss for informing me of Rudolph's result in \cite{foreman-weiss}.

\section{Preliminaries}\label{sec:prelim}

Throughout this paper, $\Ga$ always denotes a countable discrete group and $(X,\mu), (Y,\nu)$ denote standard probability spaces. We are mainly concerned with {\bf probability-measure-preserving} actions which is abbreviated as `pmp actions'. Let $\Ca$ denote the standard middle thirds Cantor set, $\Ga \cc \Ca^\Ga$ the action $(gx)(f) = x(g^{-1}f)$. This action is by homeomorphisms when $\Ca^\Ga$ is given the product topology. We let $\Prob_\Ga(\Ca^\Ga)$ denote the space of $\Ga$-invariant Borel probability measures on $\Ca^\Ga$. We give $\Prob_\Ga(\Ca^\Ga)$ the weak* topology which means that a sequence $\{\mu_n\}$ converges to a measure $\mu$ if and only if $\int f~\mu_n \to \int f~d\mu$ for every continuous function $f$ on $\Ca^\Ga$. In this topology,  $\Prob_\Ga(\Ca^\Ga)$  is compact and metrizable (by the Banach-Alaoglu Theorem). When discussing measures $\mu \in \Prob_\Ga(\Ca^\Ga)$ we say such a measure is essentially free, ergodic or has zero Rokhlin entropy to mean that the associated action $\Ga \cc (\Ca^\Ga,\mu)$ is essentially free, ergodic or has zero Rokhlin entropy.

Given a topological space $X$, a subset $Y \subset X$ is a $G_\delta$ if it can be expressed as a countable intersection of open sets. A subset $Y \subset X$ is {\bf residual} in $X$ if it contains a dense $G_\delta$ subset. If $X_0 \subset X$ then the statement `the generic element of $X$ is contained in $X_0$' means that $X_0$ is residual. 

All functions, partitions and actions considered in this paper are measurable unless explicitly stated otherwise. If $\cP$ is a partition of a measure space $(X,\mu)$, $\Ga \cc (X,\mu)$ is a pmp action and $T \subset \Ga$ is finite then $\cP^T: = \bigvee_{t\in T} t^{-1}\cP$ is the coarsest partition containing $t^{-1}\cP$ for all $t\in T$. If $T$ is infinite then $\cP^T$ is the smallest sigma-algebra containing $t^{-1}\cP$ for all $t\in T$. 

Let $\cB_X$ denote the Borel sigma-algebra on $X$. If $\cF \subset \cB_X$ is a sigma-algebra and $\cP$ is a partition then the {\bf Shannon entropy of $\cP$ relative to $\cF$} is
$$H_\mu(\cP|\cF) = \int -\log \E[\chi_{\cP(x)}|\cF](x)~d\mu(x)$$
where $\cP(x)$ denotes the part of $\cP$ containing $x$, $\chi_{\cP(x)}$ denotes the characteristic function of $\cP(x)$ and $\E[\chi_{\cP(x)}|\cF]$ denotes the conditional expectation of $\chi_{\cP(x)}$ with respect to $\cF$.

\section{Rokhlin entropy}\label{sec:rokhlin}

For any subcollection $\cF \subset \cB_X$, we let $\sigma\textrm{-alg}(\cF) \subset \cB_X$ denote the sub-sigma-algebra generated by $\cF$ and, if  $\Ga \cc X$ is a measurable action then we let $\sigma\textrm{-alg}_\Ga(\cF)$ denote the smallest sub-sigma-algebra containing $gF$ for every $g \in \Ga$ and $F \in \cF$. We do not distinguish between sigma-algebras that agree up to null sets. Thus we write $\cF_1=\cF_2$ if $\cF_1$ and $\cF_2$ agree up to null sets.

\begin{defn}
The {\bf Rokhlin entropy} of an ergodic pmp action $\Ga \cc (X,\mu)$ is defined by
$$h^{Rok}_\Ga(X,\mu) = \inf_\cP H_\mu(\cP)$$
where the infimum is over all partitions $\cP$ with $\sigma\textrm{-alg}_\Ga(\cP)=\cB_X$. For any $\Ga$-invariant $\cF \subset \cB_X$  the {\bf relative Rokhlin entropy} is defined by
$$h^{Rok}_\Ga(X,\mu|\cF)= \inf_\cP H_\mu(\cP|\cF)$$
where the infimum is over all partitions $\cP$ such $\sigma\textrm{-alg}_\Ga(\cP \cup \cF) = \cB_X$.  If $\Ga \cc (X,\mu)$ is nonergodic then the Rokhlin entropy is defined by
$$h^{Rok}_\Ga(X,\mu) = \inf_\cP H_\mu(\cP|\textrm{Inv})$$
where $\textrm{Inv}$ is the sigma-algebra of $\Ga$-invariant Borel sets. Given a collection $\cC$ of Borel subsets of $X$ then the {\bf outer Rokhlin entropy} relative to $\cF$ is defined by
$$h^{Rok}_{\Ga,\mu}(\cC|\cF) = \inf_\cP H_\mu(\cP|\cF)$$
where the infimum is over all partitions $\cP$ such that $\cC \subset \sigma\textrm{-alg}_\Ga(\cP) \vee \cF$. We also write $h^{Rok}_{\Ga,\mu}(\cC)$ instead of $h^{Rok}_{\Ga,\mu}(\cC|\cF)$ when $\cF$ is trivial. These notions were introduced and studied by B. Seward in the series \cite{seward-kreiger-1, seward-kreiger-2}. 
\end{defn}



 
 \begin{lem}\label{lem:g-delta0}
The subset of ergodic measures in $\Prob_\Ga(\Ca^\Ga)$ is a $G_\delta$ set. 
\end{lem}
 
 \begin{proof}
 This is well-known. Here is a short proof for the reader's convenience. Fix a metric $d$ on $\Prob_\Ga(\Ca^\Ga)$. For $n=1,2,3,\ldots$, let $X_n$ be the set of all measures $\mu \in  \Prob_\Ga(\Ca^\Ga)$ such that there exist measures $\mu_1,\mu_2 \in \Prob_\Ga(\Ca^\Ga)$ with $d(\mu_1,\mu_2) \ge 1/n$ and $\mu= \frac{\mu_1+\mu_2}{2}$. So $X_n$ is a closed subset and $\cup_{n=1}^\infty X_n$ is an $F_\sigma$ set. The lemma now follows from the fact that the subset of ergodic measures is the complement of $\cup_{n=1}^\infty X_n$.
 \end{proof}

Next we prove that the set of ergodic measures in $\Prob_\Ga(\Ca^\Ga)$ with zero Rokhlin entropy form a $G_\delta$ subset. For the next three lemmas we assume $\Ga \cc (X,\mu)$ is an ergodic pmp action and $\cP,\cQ$ are measurable partitions of $X$ with finite Shannon entropy.

\begin{lem}\label{lem:estimate}
$$h^{Rok}_{\Ga,\mu}(\cP) \le H_\mu(\cQ) + H_\mu(\cP | \sigma\textrm{-alg}_\Ga(\cQ)).$$
\end{lem}

\begin{proof}
Corollary 2.6 of \cite{seward-kreiger-2} implies
$$h^{Rok}_{\Ga,\mu}(\cP) \le h^{Rok}_{\Ga,\mu}(\cQ) +  h^{Rok}_{\Ga,\mu}(\cP| \sigma\textrm{-alg}_\Ga(\cQ)) \le H_\mu(\cQ) + H_\mu(\cP | \sigma\textrm{-alg}_\Ga(\cQ)).$$
\end{proof}

Let $\sP(X,\mu)$ denote the set of all partitions of $(X,\mu)$ with finite Shannon entropy in which we identify partitions that agree up to measure zero. Given partitions $\cP,\cQ \in \sP(X,\mu)$ define
$$d^{Rok}(\cP,\cQ):= H_\mu(\cP|\cQ)+H_\mu(\cQ|\cP).$$
This is the {\bf Rokhlin metric}. It is a complete separable metric on $\sP(X,\mu)$. 
 
\begin{lem}\label{lem:outer1}
Let $\sD$ be a dense subset of $\sP(X,\mu)$. Then
$$h^{Rok}_{\Ga,\mu}(\cP) = \sup_{\epsilon>0} \inf \{  H_\mu(\cQ):~ \cQ \in \sD, ~H_\mu(\cP|\sigma\textrm{-alg}_\Ga(\cQ)) <  \epsilon\}.$$
\end{lem} 

\begin{proof}
The inequality $\le$ follows from Lemma \ref{lem:estimate}. To see the opposite inequality, let $\epsilon>0$ and let $\cS$ be a partition with $\cP \subset \sigma\textrm{-alg}_\Ga(\cS)$ and $H_\mu(\cS) \le h^{Rok}_{\Ga,\mu}(\cP) + \epsilon$. Since $H_\mu(\cP| \cS^\Ga)=0$, there exists a finite subset $F \subset \Ga$ such that $H_\mu(\cP | \cS^F) < \epsilon/2.$ Since $\sD$ is dense, there exists a partition $\cR \in \sD$ such that $d^{Rok}(\cR,\cS)<\epsilon |F|^{-1}/2$. Since
$$H_\mu(\cR^F|\cS^F) \le \sum_{f\in F} H_\mu(f^{-1}\cR| \cS^F) \le  \sum_{f\in F} H_\mu(f^{-1}\cR| f^{-1}\cS) = |F|H_\mu(\cR|\cS),$$
$$d^{Rok}(\cR^F,\cS^F) \le |F| d^{Rok}(\cR,\cS) < \epsilon/2.$$
Therefore,
$$H_\mu(\cP | \sigma\textrm{-alg}_\Ga(\cR)) \le H_\mu(\cP | \cR^F) \le H_\mu(\cP | \cS^F) + d^{Rok}(\cS^F,\cR^F) < \epsilon.$$
It follows that  
 $$\inf \{  H_\mu(\cQ):~ \cQ \in \sD, ~H(\cP|\sigma\textrm{-alg}_\Ga(\cQ)) <  \epsilon\} \le H_\mu(\cR) \le H_\mu(\cS) + d^{Rok}(\cR,\cS) \le h^{Rok}_{\Ga,\mu}(\cP) +2 \epsilon.$$
 The Lemma follows by taking the limit as $\epsilon \searrow 0$ on both sides. 
\end{proof}

\begin{lem}\label{lem:outer2}
Suppose $\cP_1 \le \cP_2\le \cdots$ are an increasing sequence of partitions of $(X,\mu)$ with finite Shannon entropy such that $\bigvee_n \cP_n$ is the Borel sigma-algebra. Then $h^{Rok}_\Ga(X,\mu) =0$ if and only if $h^{Rok}_{\Ga,\mu}(\cP_n)=0$ for all $n$.
\end{lem}

\begin{proof}
The definitions of Rokhlin and outer Rokhlin entropy imply $h^{Rok}_\Ga(X,\mu) \ge  h^{Rok}_{\Ga,\mu}(\cP_n)$ for every $n$. This proves one implication. To see the other, suppose $h^{Rok}_{\Ga,\mu}(\cP_n)=0$ for all $n$. Let $\epsilon>0$. For every $n$, there exists a partition $\cQ_n$ of $X$ such that $H_\mu(\cQ_n)<\epsilon 2^{-n}$ and $\cP_n \subset \sigma\textrm{-alg}_\Ga(\cQ_n)$. Therefore, $\bigvee_n \cQ_n$ is generating and has entropy $<\epsilon$. This shows $h^{Rok}_\Ga(X,\mu) < \epsilon$. Since $\epsilon$ is arbitrary, $h^{Rok}_\Ga(X,\mu)  = 0$. 
 \end{proof}

\begin{lem}\label{lem:g-delta1}
The set
$$E_0:=\{ \mu \in \Prob_\Ga(\Cantor^\Ga):~ h^{Rok}_\Ga(\Ca^\Ga,\mu) =0~ \textrm{and}~\Ga \cc (\Ca^\Ga,\mu)~\textrm{ergodic}\}$$
is a $G_\delta$ set.
\end{lem}

\begin{proof}

Let $\cP_n$ be an increasing sequence of finite partitions of $\Ca^\Ga$ such that all elements of $\cP_n$ are clopen (=closed and open) and $\bigvee_n \cP_n$ is the full Borel sigma-algebra. Let
$$E_n:=\{ \mu \in \Prob_\Ga(\Cantor^\Ga):~ h^{Rok}_{\Ga,\mu}(\cP_n) =0~ \textrm{and}~\Ga \cc (\Ca^\Ga,\mu)~\textrm{ergodic}\}.$$
By Lemma \ref{lem:outer2}, $E_0 = \cap_n E_n$. So it suffices to show each $E_n$ is a $G_\delta$. Let $\sD$ denote the collection of clopen partitions of $\Ca^\Ga$. Then $\sD$ is dense in $\sP(\Ca^\Ga,\mu)$ for every Borel probability measure $\mu$. For any $\cQ \in \sD$ and finite $F \subset \Ga$, the maps $\mu \mapsto H_\mu(\cQ)$ and $\mu \mapsto H_\mu(\cP_n | \cQ^F)$ are continuous (because all partitions involved are clopen). So for any $\epsilon>0$, the set
$$\{ \mu \in \Prob_\Ga(\Cantor^\Ga):~ H_\mu(\cP_n | \cQ^F) < \epsilon\}$$
is open. Let $\sO(\epsilon)$ denote the set of all measures $\mu \in \Prob_\Ga(\Ca^\Ga)$ such that there exist $\cQ \in \sD$ and finite $F \subset \Ga$ with $H_\mu(\cQ)<\epsilon$ and $H_\mu(\cP_n | \cQ^F) < \epsilon$. Then $\sO(\epsilon)$ is open. By Lemma \ref{lem:outer1},
$$E_n  = \bigcap_{m=1}^\infty \sO(1/m) \cap \{\mu:~ \Ga \cc (\Ca^\Ga,\mu)~\textrm{ergodic}\}.$$
By Lemma \ref{lem:g-delta0}, this implies $E_n$ is a $G_\delta$. 

\end{proof}

\section{A zero entropy action that surjects every Bernoulli shift}\label{sec:inverse-limit}       

Bernoulli shifts are defined as follows: let $(K,\kappa)$ denote a standard probability space and $(K,\kappa)^\Ga$ the product measure space. Let $\Ga$ act on $K^\Ga$ by $(gx)(f)=x(g^{-1}f)$ for $x \in K^\Ga, g,f \in \Ga$. This action is measure-preserving and is called the {\bf Bernoulli shift over $\Ga$ with base $(K,\kappa)$}.

This section constructs a zero Rokhlin entropy action that factors onto all Bernoulli shifts (assuming $\Ga$ is nonamenable). The main part of the argument is in the next proposition: that there are factors of Bernoulli shifts with small entropy that factor onto all Bernoulli shifts. 

\begin{prop}\label{prop:small-ext}
Let $\Ga$ be a countable nonamenable group. Then for every $\epsilon>0$ there exists a pmp action $\Ga \cc (Y,\nu)$ satisfying:
\begin{itemize}
\item $\Ga \cc (Y,\nu)$ is a factor of a Bernoulli shift,
\item $h^{Rok}_\Ga(Y,\nu)<\epsilon$,
\item $\Ga \cc (Y,\nu)$ factors onto every Bernoulli shift over $\Ga$.
\end{itemize}
\end{prop}

The proof uses the fact that, for nonabelian free groups, all Bernoulli shifts factor onto each other. In order to apply this we need some concepts from measured equivalence relations. So: given an action $\Ga\cc (X,\mu)$ the {\bf orbit-equivalence relation} is the relation $\cR_\Ga:=\{(x,\ga x) \in X \times X:~x\in X, \ga \in \Ga\}$. A {\bf subequivalence relation} is any measurable subset $\cS \subset \cR_\Ga$ that is an equivalence relation in its own right. It is {\bf finite} if for almost every $x\in X$ the $\cS$-class of $x$ is finite. It {\bf hyperfinite} if there exists an increasing sequence $\cS_1 \subset \cS_2 \subset \cdots$ of finite subequivalence relations such that $\cS= \cup_i \cS_i$. A subset $Y \subset X$ is {\bf $\cS$-saturated} if $Y$ is a union of $\cS$-equivalence classes. The subequivalence $\cS$ is {\bf ergodic} if every measurable $\cS$-saturated subset is either null or co-null. A {\bf graphing} of $\cS$ is a subset $\cG \subset\cS$ such that 
\begin{itemize}
\item $(x,y) \in \cG \Rightarrow (y,x) \in \cG$;
\item for every $(x,y) \in \cS$ there exists $x=x_0,x_1,\ldots, x_n=y$ such  $(x_i,x_{i+1}) \in \cG$ for all $0\le i <n$.
\end{itemize}
A graphing $\cG$ determines a graph with vertex set $X$ and edges consisting of unordered pairs $\{x,y\}$ such that $(x,y) \in \cG$. If the connected components of this graph are trees then $\cG$ is called a {\bf treeing} and $\cS$ is said to be {\bf treeable}. Intuitively, graphings are treated in a manner similar to Cayley graphs and treeable subequivalence relations are analogous to free subgroups.

\begin{lem}\label{lem:GL}
Let $\Ga \cc (X,\mu)$ be an essentially free factor of a Bernoulli shift and suppose that its orbit-equivalence relation contains a non-hyperfinite treeable subequivalence relation $\cS$. Then for every pair of probability spaces $(K,\kappa), (L,\lambda)$ the direct product action 
$$\Ga \cc (X \times K^\Ga, \mu \times \kappa^\Ga)$$
factors onto the Bernoulli shift $\Ga \cc (L,\lambda)^\Ga$. 
\end{lem}

\begin{proof}
I claim that we can choose $\cS$ to be ergodic. Since $\cS$ is non-hyperfinite, $\Ga$ must be nonamenable. Then the main result of \cite{MR2647134} implies that there exists a measurable subset $Y \subset X$ with positive measure such that $\cS$ restricted to $Y$ is ergodic. Let $\phi:X \to Y$ be any measurable map such that (a) the graph of $\phi$ is contained in the orbit-equivalence relation and (b) $\phi$ restricted to $Y$ is the identity map. Now let $\cS'$ be the equivalence relation given by $(x,y) \in \cS'$ if and only if $(\phi x, \phi y) \in \cS$. This is a subequivalence relation of the orbit-equivalence relation; it is ergodic because any nonnull $\cS'$-invariant measurable subset necessarily contains $Y$ (since $\cS$ is ergodic and $\cS' \cap Y\times Y = \cS$) and therefore contains $X$ (up to measure zero). It is also treeable. Indeed if $\cG$ is a treeing of $\cS \cap Y \times Y$ then we define a treeing $\cG'$ of $\cS'$ by $\cG' = \cG \cup \{ (x,\phi(x)), (\phi(x),x):~ x\in X - Y\}$. So we can choose $\cS$ to be ergodic.

By \cite[Prop. 14]{gaboriau-lyons} the existence of an ergodic non-hyperfinite treeable subequivalence relation implies the existence of an essentially free ergodic pmp action $\F_2 \cc (X,\mu)$ of the rank 2 free group whose orbits are contained in $\Ga$-orbits (the main part of this argument is due to Hjorth \cite{hjorth-cost-attained}). Let $c:\F_2 \times X \to \Ga$ denote the cocycle
$$c(f,x)=g \Leftrightarrow fx =gx.$$
Also, for $x \in X$ and $y \in K^\Ga$ define $F_x(y) \in K^{\F_2}$ by
$$F_x(y)(f) = y(c(f^{-1},x)^{-1}).$$

By \cite{bowen-ornstein-2011} there exists a factor map $\Phi:(K,\kappa)^{\F_2} \to (L,\lambda)^{\F_2}$. So we define $\Psi:X \times K^\Ga \to L^\Ga$ by 
$$\Psi(x,y)(\gamma) = \Phi(F_{\gamma^{-1} x}(\gamma^{-1} y))(1_{\F_2}).$$
It is routine to check that this is the required factor. For the sake of clarity, here is an explanation without the algebra. An element $x\in X$ has the property that its $\Ga$-orbit is partitioned into $\F_2$-orbits. We consider an element $y \in K^\Ga$ as a coloring of $\Ga$ with colors in $K$. By identifying $\Ga$ with the orbit of $x$, we may also think of $y$ as a coloring of the orbit of $x$. This coloring does not change if we replace the pair $(x,y)$ with $(gx,gy)$ for $g\in \Ga$. By restriction, we can also view $y$ as a coloring of the $\F_2$-orbits that make up the $\Ga$-orbit of $x$. By identifying each $\F_2$-orbit with $\F_2$ itself we can view $y$ as a coloring of $\F_2$ (actually several copies of $\F_2$, one for each $\F_2$-orbit making up the $\Ga$-orbit). We can apply $\Phi$ to such a coloring to obtain a new coloring of (several copies of) $\F_2$ with values in $L$. By identifying each such copy of $\F_2$ with the $\F_2$-orbits in $\Gamma x$, we obtain again a coloring of the $\F_2$-orbits of $x$ contained in the $\Gamma$-orbit of $x$ and therefore, we obtain a coloring of $\Gamma$ by $L$. This is what the map $\Psi$ does.
\end{proof}

In order to use the lemma above to prove Proposition \ref{prop:small-ext}, we need to construct the factor $\Ga \cc (Y,\nu)$ in such a way that its orbit equivalence relation contains a non-hyperfinite treeable subequivalence relation. This will be accomplished through percolation theory for which we will need a bit of background. So let $G=(V,E)$ be a graph and $p \in [0,1]$ a parameter. The {\bf Bernoulli bond percolation} with parameter $p$ is the random subset $\omega_p \subset E$ defined by: if $e \in E$ is an edge then $e \in \omega_p$ with probability $p$. Moreover the events $\{e\in \omega_p:~e\in E\}$ are jointly independent. This is also called {\bf $p$-bond percolaton}. We consider $\omega_p$ to be a random subgraph of $G$. A {\bf cluster} is a connected component of $\omega_p$. The {\bf critical bond percolation} of $G$ is the number $p_c(G)$ equal to the infimum over all $p>0$ such that Bernoulli bond percolation with parameter $p$ has an infinite cluster almost surely. See \cite{blps-group-perc} for background.

\begin{lem}\label{lem:alpha}
Let $D>2$ be an integer. There exists $0<\alpha,\beta<1$ such that the following holds. Let $G$ be a tree such that every vertex in $G$ has degree at least 3 and at most $D$. Then $\alpha$-Bernoulli bond percolation on $G$ has an infinite cluster a.s. and every such cluster is a tree with infinitely many ends. Also, for any vertex $v$ of $G$, the probability (with respect to $\alpha$-bond percolation)  that $v$ is contained in a finite cluster is at least $\beta$.
\end{lem}

\begin{proof}
Note that $G$ contains a copy of the 3-regular tree $T_3$. Therefore $p_c(G) \le p_c(T_3)$. It is well-known that $p_c(T_3)<1$. This follows, for example, from the more general statement that $p_c(H)<1$ whenever $H$ is the Cayley graph of a nonamenable group \cite{blps-nonamenable} (observe that $T_3$ is the Cayley graph of $\Z/2\Z * \Z/2\Z * \Z/2\Z$). So let $\alpha=(p_c(T_3)+1)/2$. Let $\omega \subset E(G)$ denote $\alpha$-bond percolation on $G$. Since $G$ is a tree, $\omega$ is a forest a.s. By \cite{haggstrom-peres-1999}, each infinite cluster of $\omega$ has infinitely many ends a.s. (for a simpler proof, see \cite{lyons-schramm-indistinguishability}).

The probability that a vertex $v$  is contained in a finite cluster of $\omega$ is at least the probability that $v$ is itself a cluster. The latter probability is $(1-\alpha)^{\textrm{deg}(v)} \ge (1-\alpha)^{D} =:\beta$.  
\end{proof}

\begin{proof}[Proof of Proposition \ref{prop:small-ext}]
Let $\Ga_0\le \Ga$ be a finitely generated nonamenable subgroup. By \cite{pak-nagnibeda}, there exists a finite generating set $S \subset \Ga_0$ such that bond-percolation on the Cayley graph $\Cay(\Ga_0,S)$ has a nontrivial uniqueness phase. In other words, there exists $p \in (0,1)$ such that $p$-bond-percolation on $\Cay(\Ga_0,S)$ has infinitely many infinite clusters. It follows by inclusion that $p$-bond-percolation on $\Cay(\Ga,S)$ also has infinitely many infinite clusters. Here $\Cay(\Ga,S)$ is the graph with vertex set $\Ga$ and edges of the form $(g,gs)$ for $g\in \Ga, s\in S$. This need not be a connected graph since $S$ need not generate $\Ga$.  

Let $\omega_0 \subset E(\Cay(\Ga,S))$ denote the set of edges of $p$-bond-percolation on $\Cay(\Ga,S)$. By \cite{haggstrom-peres-1999}, each infinite cluster of $\omega_0$ has infinitely many ends a.s. (for a simpler proof, see \cite{lyons-schramm-indistinguishability}). For $x\in \Ga$, let $K_0(x)$ denote the cluster of $\omega_0$ containing $x$. 

By \cite[Lemma 7.4]{blps-group-perc}, there exists a percolation $\omega_1 \subset \omega_0$ such that conditioned on the cluster $K_0(x)$ being infinite, the cluster $K_1(x)$ of $\omega_1$ containing $x$ is a tree with infinitely many ends (almost surely). Moreover the proof shows that we can choose $\omega_1$ to be the minimal spanning forest associated with an iid process. In particular, we can choose $\omega_1$ so that its law is a factor of a Bernoulli process. After removing some edges if necessary, we may also assume that every finite cluster of $\omega_1$ consists of a single vertex.

Let $\alpha,\beta$ be as in Lemma \ref{lem:alpha}. 

{\bf Claim}. There exist random subgraphs $\omega_1 \supset \omega_2 \supset \cdots$ satisfying:
\begin{itemize}
\item each infinite cluster of $\omega_i$ is a tree with infinitely many ends (a.s.),
\item every finite cluster of $\omega_i$ is a single vertex,
\item the probability that $1_\Ga$ is contained in an infinite cluster of $\omega_{i+1}$ is at most $\beta$ times the probability that $1_\Ga$ is contained in an infinite cluster of $\omega_i$. 
\item each $\omega_{i}$ is a factor of a Bernoulli shift.
\end{itemize}

\begin{proof}
For induction, we assume $\omega_1,\ldots, \omega_n$ has been constructed.

We cannot directly apply Lemma \ref{lem:alpha} because some vertex might have degree $<3$ in $\omega_n$. After repeatedly removing all edges incident to a degree 1 vertex if necessary, we may assume that no vertex of $\omega_n$ has degree 1. Next define $\omega'_n$ as follows: the vertices of $\omega'_n$ are the vertices of $\omega_n$ that have degree at least 3. There is an edge in $\omega'_n$ from $v$ to $w$ if there is a path in $\omega_n$ from $v$ to $w$ such that all of the intermediate vertices have degree 2. 

Let $\omega'_{n+1} \subset \omega'_n$ be the random subgraph obtained from  Bernoulli $\alpha$-bond-percolation on $\omega'_n$. By Lemma \ref{lem:alpha}, $\omega'_{n+1}$ contains infinite clusters a.s. Moreover each infinite cluster is a tree with infinitely many ends (since each infinite cluster of $\omega'_n$ is a tree with infinitely many ends). We let $\omega''_{n+1}$ be the subgraph of $\omega_n$ that is induced from $\omega'_{n+1}$. More precisely, recall that every edge $e$ of $\omega'_{n+1}$ corresponds to a path $e_1,e_2,\ldots,e_k$ of edges in $\omega_n$ such that each intermediate vertex has degree 2. We let $\omega''_{n+1}$ be the subgraph containing all such edges $e_1,\ldots, e_k$. Finally we let $\omega_{n+1}$ be the subgraph obtained from $\omega''_{n+1}$ by removing all edges that are contained in finite clusters. The properties in the claim are easily verified for $\omega_{n+1}$. This completes the induction.
\end{proof}

Let $\delta>0$ be such that
$$\delta |S|\log (2) - \delta\log(\delta) - (1-\delta)\log(1-\delta)<\epsilon/2.$$
It follows from the claim above that there exists a random subgraph $\omega_n$ of $\Cay(\Ga,S)$ (for some $n$) such that:
\begin{itemize}
\item the probability that $\omega_n$ does not contain any edges incident to $1_\Ga$  is at least $1-\delta$,
\item the law of $\omega_n$ is a factor of a Bernoulli shift,
\item with probability one, some cluster of $\omega_n$ is a tree with infinitely many ends.
\end{itemize}

Let $X$ be the space of all subgraphs of $\Cay(\Ga,S)$ and $\mu$ the law of $\omega_n$. For $x\in X$, let $\phi(x) =\{s\in S:~(1_\Ga,s)\in \omega_n\}$. Let $\cP$ be the partition of $X$ induced by $\phi$: this means that $x,y \in X$ are in the same part of $\cP$ if and only if $\phi(x)=\phi(y)$. The Shannon entropy of $\cP$ satisfies the bound:
$$H_\mu(\cP)\le \delta |S|\log (2) - \delta\log(\delta) -  (1-\delta)\log(1-\delta)<\epsilon/2$$
(because there are $2^{|S|}$ subsets of $S$ and the probability that $\phi(x)$ is empty (when $x\in X$ is random with law $\mu$)  is at least $1-\delta$). The partition $\cP$ is generating for $\Ga \cc (X,\mu)$. Therefore
$$h^{Rok}_\Ga(X,\mu)<\epsilon/2.$$
Because each $\omega_n$ contains an infinite tree with infinitely many ends, the orbit-equivalence relation of $\Ga \cc (X,\mu)$ contains a non-hyperfinite treeable subequivalence relation. To see this, let $Y \subset X$ be the set of all $\omega \in X$ such that $1_\Ga$ is in an infinite cluster of $\omega$. Let $\cF \subset Y \times Y$ be the Borel equivalence relation on $Y$ given by $(g\omega, \omega) \in \cF$ if and only if $g^{-1}$ and $1_\Ga$ are in the same infinite cluster of $\omega$. This is a non-hyperfinite treeable equivalence relation since its equivalence classes are in 1-1 bijection with the infinite clusters of $\omega$.  Let $\Phi:X \to Y$ be any Borel map with graph contained in the orbit-equivalence relation of $\Ga$ such that $\Phi$ restricted to $Y$ is the identity map. Finally let $\tcF \subset X \times X$ be the equivalence relation $(x,y) \in \tcF$ if and only if $(\Phi x, \Phi y) \in \cF$. Then $\tcF$ is the required non-hyperfinite treeable subequivalence relation. In fact, if $\cG \subset Y \times Y$ is a treeing of $\cF$ then $\tcG:=\cG \cup \{(x,\Phi(x)):~x\in X\}$ is a treeing of $\tcF$.

If $\Ga \cc (X,\mu)$ is not essentially free then let $(L,\lambda)$ be a nontrivial probability space with Shannon entropy small enough so that the Rokhlin entropy of the direct product $\Ga \cc (X\times L^\Ga, \mu \times \lambda^\Ga)$ is $< \epsilon/2$. Because $\Ga \cc (X,\mu)$ is a factor of a Bernoulli shift, this direct product is also a factor of a Bernoulli shift. Moreover, it is essentially free. Also its orbit-equivalence relation contains a non-hyperfinite treeable subequivalence relation (this can be obtained by pulling back a non-hyperfinite treeable subequivalence relation of $\Ga \cc (X,\mu)$ by way of the projection map). So without loss of generality, we may assume $\Ga \cc (X,\mu)$ is essentially free.

Let $(K,\kappa)$ be any nontrivial probability space with Shannon entropy $<\epsilon/2$. Lemma \ref{lem:GL} now implies that the product action
$$\Ga \cc (X\times K^\Ga, \mu \times \kappa^\Ga)$$
satisfies the statement of the Theorem. 

\end{proof}

\begin{cor}\label{cor:inverse limit}
Let $\Ga$ be any countable nonamenable group. There exists a pmp action $\Ga \cc (Z,\zeta)$ satisfying:
\begin{itemize}
\item $\Ga \cc (Z,\zeta)$  is an inverse limit of factors of Bernoulli shifts,
\item $h^{Rok}_\Ga(Z,\zeta)=0$
\item $\Ga \cc (Z,\zeta)$ factors onto all Bernoulli shifts over $\Ga$.
\end{itemize}
\end{cor}

\begin{proof}
By Proposition \ref{prop:small-ext} there exists a sequence $\Ga \cc (Y_i,\nu_i)$ ($i\in \N$) of pmp actions satisfying
\begin{itemize}
\item each $\Ga \cc (Y_i,\nu_i)$  is a factor of a Bernoulli shift,
\item $h^{Rok}_\Ga(Y_i,\nu_i)<2^{-i}$,
\item each $\Ga \cc (Y_i,\nu_i)$ factors onto all Bernoulli shifts over $\Ga$.
\end{itemize}
It follows that there exist factor maps $\Phi_i:Y_i \to Y_{i-1}$ for $i\ge 2$. Let $\Ga \cc (Z,\zeta)$ denote the inverse limit of this system. It suffices to show $h^{Rok}_\Ga(Z,\zeta)=0$. This follows from \cite[Corollary 4.9]{seward-kreiger-2}. Alternatively, it can be proven directly as follows. Let $\epsilon>0$. Then there exists an infinite subsequence $\{n_i\}_{i=1}^\infty$ such that 
$$\sum_i h^{Rok}_\Ga(Y_{n_i},\nu_{n_i}) < \epsilon/2.$$
Let $\cP_i$ be a generating partition of $Y_{n_i}$ with $H_\mu(\cP_i) < h^{Rok}_\Ga(Y_{n_i},\nu_{n_i}) + \epsilon 2^{-i-1}$. By pulling back, we may consider $\cP_i$ to be a partition of $Z$. Then $\bigvee_i \cP_i$ is a generating partition for $\Ga \cc (Z,\zeta)$ and 
$$H_\mu\left(\bigvee_i \cP_i \right) \le \sum_i H_\mu(\cP_i) \le \sum_i h^{Rok}_\Ga(Y_{n_i},\nu_{n_i}) + \epsilon 2^{-i-1} < \epsilon.$$
Because $\epsilon>0$ is arbitrary, this proves $h^{Rok}_\Ga(Z,\zeta)=0$.

\end{proof}

\section{Zero entropy extensions}\label{sec:extension} 

\begin{thm}\label{thm:zero-extension}
Let $\Ga$ be a nonamenable countable group and $\Ga \cc (X,\mu)$ a free ergodic action. Then there exists a free ergodic action $\Ga \cc (\tX,\tmu)$ that factors onto $\Ga \cc (X,\mu)$ and has zero Rokhlin entropy.
\end{thm}

\begin{remark}
 Seward \cite{seward-small-action} proved, under the same hypotheses as Theorem \ref{thm:zero-extension}, the existence of an extension $\Ga \cc (\tX,\tmu)$ of $\Ga \cc (X,\mu)$ such that $\Ga \cc (\tX,\tmu)$ admits a generating partition with at most $n$ parts where $n=n(\Ga)$ depends only on $\Ga$. By Seward's generalization of Krieger's Generator Theorem \cite{seward-kreiger-1}, Theorem \ref{thm:zero-extension} implies that we can take $n=2$.
\end{remark}


We will need Seward's generalization of Sinai's Factor Theorem  \cite{seward-sinai}:

\begin{thm}[Seward \cite{seward-sinai}]\label{thm:seward-sinai}
For any countable group $\Ga$ and any ergodic essentially free action $\Ga \cc (X,\mu)$  with positive Rokhlin entropy there exists a Bernoulli factor such that the Rokhlin entropy of $\Ga \cc (X,\mu)$ relative to this Bernoulli factor is zero.
\end{thm}

\begin{proof}[Proof of Theorem \ref{thm:zero-extension}]
Without loss of generality, we may assume $\Ga \cc (X,\mu)$ has positive Rokhlin entropy. By Theorem  \ref{thm:seward-sinai}, there exists a Bernoulli factor $\Ga \cc (B,\beta)$ of $\Ga\cc (X,\mu)$ such that 
$$h^{Rok}_\Ga(X,\mu | \cB_{B}) = 0$$
where $\cB_{B}$ denotes the sigma-algebra associated with $B$. Let $\Ga \cc (Z,\zeta)$ be as in Corollary \ref{cor:inverse limit}. Fix a factor map of  $\Ga \cc (Z,\zeta)$ onto $\Ga \cc (B,\beta)$. Let $\Ga \cc (\tX,\tmu)$ be the independent joining of $\Ga \cc (Z,\zeta)$ and $\Ga \cc (X,\mu)$ over $\Ga \cc (B,\beta)$. 

It suffices to show $h^{Rok}_\Ga(\tX,\tmu)=0$. By \cite[Corollary 2.6]{seward-kreiger-2},
$$h^{Rok}_\Ga(\tX,\tmu) \le h^{Rok}_\Ga(\tX,\tmu|\cB_{B}) + h^{Rok}_{\Ga,\tmu}(\cB_B).$$
Because outer Rokhlin entropy is upper-bounded by the Rokhlin entropy of any intermediate factor,
$$h^{Rok}_{\Ga,\tmu}(\cB_B) \le h^{Rok}_\Ga(Z,\zeta) = 0.$$
So it suffices to prove  $h^{Rok}_\Ga(\tX,\tmu|\cB_{B})=0$. 

Let $\epsilon>0$, $\alpha$ be a generating partition of $Z$ with $H_\zeta(\alpha)<\epsilon$ and let $\beta$ be a partition of $X$ with $H_\mu(\beta|\cB_B)<\epsilon$ such that $\sigma\textrm{-alg}_\Ga(\beta \cup \cB_{B}) = \cB_X$ (up to measure zero). By pulling back, we may consider $\alpha$ and $\beta$ as partitions on $\tX$. Clearly, $\alpha \vee \beta$ is generating for the action $\Ga \cc (\tX,\tmu)$ and $H_{\tmu}(\alpha\vee \beta|\cB_B)<2\epsilon$. Since $\epsilon>0$ is arbitrary, this implies the claim.

\end{proof}

\section{Zero entropy is generic}\label{sec:zero} 

In this section, the proof of Theorem \ref{thm:zero} is completed. Most of our results so far hold only for essentially free ergodic actions. In order to generalize them, first we show that essentially free actions are generic. The next lemma will be helpful twice.

\begin{lem}\label{lem:2}
Let $\bfa=\Ga \cc (X,\mu)$ be a pmp action and $\Phi:X \to \Ca^\Ga$ a $\Ga$-equivariant measurable map. Then there exists a sequence of measures $\mu_i \in \Prob_\Ga(\Ca^\Ga)$ such that
\begin{itemize}
\item $\Ga \cc (\Ca^\Ga,\mu_i)$ is measurably-conjugate to $\bfa$ for all $i$;
\item $\mu_i \to \Phi_*\mu$ in the weak* topology as $i\to\infty$.
\end{itemize}
\end{lem}

\begin{proof}
Let $\Psi:X \to \Cantor^\Ga$ be a $\Ga$-equivariant measurable map such that $\Ga \cc (\Ca^\Ga,\Psi_*\mu)$ is measurably conjugate to $\bfa$. To see that such a map exists, identify $\Ca$ with $\{0,1\}^\N$ (where the latter has the product topology). We consider an element $x\in \{0,1\}^\N$ to be a function $x:\N \to \{0,1\}$. Choose a sequence $\psi_i:X \to \{0,1\}$ of measurable maps such that for all distinct elements $x,y \in X$ there exists some $i$ such that $\psi_i(x) \ne \psi_i(y)$. Then define $\Psi(x)(1_\Ga)(n) = \psi_n(x)$ and in general, define $\Psi(x)(g)=\Psi(g^{-1}x)(1_\Ga)$. It is routine to check that this satisfies the claim.

 Define $\Ga$-equivariant maps $\Phi_{n}:X \to \Cantor^\Ga$ so that the first $n$-coordinates of $\Phi_n(x)$ agree with those of $\Phi(x)$ and the last coordinates agree with $\Psi(x)$. In other words, for every $g\in \Ga$,
$$\Phi_n(x)(g) = (\Phi(x)(g)(1),\ldots, \Phi(x)(g)(n), \Psi(x)(g)(1), \Psi(x)(g)(2),\ldots).$$
As above we are identifying $\Ca$ with $\{0,1\}^\N$. Clearly, $\Phi_n$ is $\Ga$-equivariant, is an isomorphism onto its image and $\lim_{n\to\infty} \Phi_{n*}\mu  = \Phi_*\mu$. To finish the lemma, set $\mu_i:=\Phi_{n*}\mu$. 
\end{proof}

Let $\Prob^{erg}_\Ga(\Ca^\Ga) \subset \Prob_\Ga(\Ca^\Ga)$ denote the subset of ergodic measures.

\begin{lem}\label{lem:ess-free}
The subset of all essentially free measures in $\Prob_\Ga(\Ca^\Ga)$ is a $G_\delta$ set. Moreover, this subset is dense in $\Prob_\Ga(\Ca^\Ga)$ and its intersection with $\Prob^{erg}_\Ga(\Ca^\Ga)$ is dense in $\Prob^{erg}_\Ga(\Ca^\Ga)$. 
\end{lem}

\begin{proof}
For any element $g \in \Ga$, let $\Fix(g)=\{x \in \Ca^\Ga:~gx=x\}$. Then $\Fix(g)$ is compact in $\Ca^\Ga$. By the Portmanteau Theorem, for every $\epsilon>0$, the set $\{\mu \in \Prob_\Ga(\Ca^\Ga):~ \mu(\Fix(g)) < \epsilon\}$ is open. Therefore,
$$\bigcap_{g\in \Ga-\{1_\Ga\}} \bigcap_{n=1}^\infty \{\mu \in \Prob_\Ga(\Ca^\Ga):~ \mu(\Fix(g)) < 1/n\}$$
is a $G_\delta$ set. The above set is the same as the subset of essentially free measures. This proves the first claim.

To prove the second claim, let $\mu \in \Prob_\Ga(\Ca^\Ga)$ be arbitrary. We observe that the direct product of $\Ga \cc (\Ca^\Ga,\mu)$ with a Bernoulli shift is essentially free and factors onto $\Ga \cc (\Ca^\Ga,\mu)$. Moreover this product is ergodic if $\mu$ is ergodic. So Lemma \ref{lem:2} implies that $\mu$ is a weak* limit of essentially free measures and these measures can be chosen to be ergodic if $\mu$ is ergodic.
\end{proof}

The next step shows that the generic {\em ergodic} measure has zero Rokhlin entropy. 

\begin{prop}\label{prop:ergodic}
 The subset of measures $\mu \in  \Prob^{erg}_\Ga(\Ca^\Ga)$ such that the corresponding action $\Ga \cc (\Ca^\Ga,\mu)$ is essentially free and has zero Rokhlin entropy is a dense $G_\delta$.
\end{prop}

\begin{proof}
Lemmas \ref{lem:ess-free} and \ref{lem:g-delta1} show that this subset is a $G_\delta$. If $\Ga$ is nonamenable then  it is dense by Lemmas \ref{lem:2}, \ref{lem:ess-free}   and Theorem \ref{thm:zero-extension}. If $\Ga$ is amenable then the result is due to Rudolph (see the Subclaim after Claim 19 in \cite{foreman-weiss}). This uses the fact that Rokhlin entropy agrees with classical entropy by \cite{seward-tucker-drob}. 

\end{proof}

Next we prove that any property that is residual for ergodic measures is automatically residual for all measures. To make this precise, let 
$$\beta: \Prob( \Prob_\Ga^{erg}(\Ca^\Ga) ) \to \Prob_\Ga(\Ca^\Ga)$$
$$\pi: \Prob_\Ga(\Ca^\Ga) \to \Prob( \Prob_\Ga^{erg}(\Ca^\Ga) )$$
denote the barycenter map and the ergodic decomposition map respectively. To be precise, 
$$\beta(\omega): = \int \mu~d\omega(\mu)$$
and $\pi$ is the inverse of $\beta$. 

\begin{prop}\label{prop:non-ergodic}
Let $\cZ_0 \subset \Prob^{erg}_\Ga(\Ca^\Ga)$ be Borel and define 
$$\cZ = \{\mu \in \Prob_\Ga(\Ca^\Ga):~ \pi(\mu)(\cZ_0)=1\}.$$
If $\cZ_0$ is residual in $\Prob^{erg}_\Ga(\Ca^\Ga)$ then $\cZ$ is residual in $\Prob_\Ga(\Ca^\Ga)$.
\end{prop}

First we need a lemma:
\begin{lem}\label{lem:gw}
The barycenter map $\beta$ is continuous. The ergodic decomposition map $\pi$ is continuous if and only if $\Ga$ has property (T) in which case it is a homeomorphism.
\end{lem}

\begin{proof}
The first statement is straightforward. The main result of \cite{glasner1997kazhdan} states that if $\Ga$ has property (T) then  $\Prob_\Ga^{erg}(\Ca^\Ga)$ is a closed (and therefore compact) subset of  $\Prob_\Ga(\Ca^\Ga)$. On the other hand, if $\Ga$ does not have (T) then $\Prob_\Ga^{erg}(\Ca^\Ga)$ is dense in $\Prob_\Ga(\Ca^\Ga)$. Since $\beta$ and $\pi$ are bijective, these two statements imply the lemma.
\end{proof}

\begin{proof}[Proof of Proposition \ref{prop:non-ergodic}]
\noindent {\bf Case 1}. Suppose $\Ga$ does not have property (T). By \cite{glasner1997kazhdan} $\Prob_\Ga^{erg}(\Ca^\Ga)$ is dense in $\Prob_\Ga(\Ca^\Ga)$. By Lemma \ref{lem:g-delta0} $\Prob_\Ga^{erg}(\Ca^\Ga)$ is a $G_\delta$. Therefore $\Prob_\Ga^{erg}(\Ca^\Ga)$ is residual in $\Prob_\Ga(\Ca^\Ga)$. So $\cZ_0$ is residual in $\Prob_\Ga(\Ca^\Ga)$. Since $\cZ_0 \subset \cZ$, this proves $\cZ$ is also residual.

\noindent {\bf Case 2}. Suppose $\Ga$ has property (T). Let 
$$\cY=\{\omega \in \Prob( \Prob_\Ga^{erg}(\Ca^\Ga) ):~ \omega(\cZ_0)=1\}.$$
By  Lemma \ref{lem:gw} it suffices to prove that $\cY$ is residual. Since $\cZ_0$ contains a dense $G_\delta$, we may assume without loss of generality that $\cZ_0$ is a dense $G_\delta$. So the portmanteau Theorem implies $\cY$ is a $G_\delta$ subset. 

Let $d$ be a continuous metric on $\Prob_\Ga^{erg}(\Ca^\Ga)$. Because $\cZ_0$ is dense in $\Prob_\Ga^{erg}(\Ca^\Ga)$ for every $n\in \N$ there exists a Borel map $\Phi_n:\Prob_\Ga^{erg}(\Ca^\Ga) \to \cZ_0$  with $d(x,\Phi_n(x))<1/n$ for all $x$. Then for every $\mu \in \Prob(\Prob_\Ga^{erg}(\Ca^\Ga))$, $\Phi_{n*}\mu \to \mu$ in the weak* topology as $n\to\infty$. Since $\Phi_{n*}\mu \in \cY$, this proves $\cY$ is dense. 
\end{proof}

\begin{proof}[Proof of Theorem \ref{thm:zero}]
The main theorem of \cite{alpeev-seward} implies that an action has zero Rokhlin if and only if almost every ergodic component has zero Rokhlin entropy. Also \cite[Corollary 4.4]{seward-kreiger-2} shows that $\cZ_0$ is Borel (where $\cZ_0 \subset \Prob_\Ga^{erg}(\Ca^\Ga)$ is the set of measures with zero Rokhlin entropy). So Theorem \ref{thm:zero} follows from Propositions \ref{prop:ergodic} and \ref{prop:non-ergodic}.
\end{proof}

\begin{remark}
Here is a brief sketch of an alternative proof of Theorem \ref{thm:zero}. Using the nonergodic version of Seward's generalization of Sinai's Theorem \cite{seward-sinai} in the proof of Theorem \ref{thm:zero-extension}, it can be shown that every essentially free pmp action admits a zero Rokhlin entropy extension (ergodicity is not required). The theory of weak equivalence of actions shows that the measure conjugacy class of any action in $A(\Ga,X,\mu)$ contains the conjugacy class of each of its factors. Because essentially free actions are dense in $A(\Ga,X,\mu)$, it follows that zero Rokhlin entropy actions are also dense in $A(\Ga,X,\mu)$.  In \cite[Lemma 8.7]{alpeev-seward}, it is proven that the subset of all zero-Rokhlin entropy actions in $A(\Ga,X,\mu)$ is a $G_\delta$ subset. Alternatively, this can be proven in a manner similar to the proof of Lemma \ref{lem:g-delta1}.
\end{remark}

\section{Naive entropy}\label{sec:naive} 

This section introduces naive entropy. The main result is that zero naive entropy is closed under factors, self-joinings and inverse limits.

 \begin{defn}
Let $\Ga \cc (X,\mu)$ be a pmp action and $\cP$ a partition of $X$. The {\bf naive entropy} of $\cP$ is 
$$h^{naive}_\mu(\cP) = \inf_{W\subset  \subset \Ga} |W|^{-1} H_\mu(\cP^W)$$
where $\subset \subset$ means ``a finite subset of''. The {\bf naive entropy} of $\Ga \cc (X,\mu)$ is 
$$h^{naive}_\Ga(X,\mu) = \sup_\cP h^{naive}_\mu(\cP)$$
where the supremum is over all finite-entropy partitions $\cP$.
\end{defn}

It is an exercise to show that if $\Ga$ is amenable then naive entropy coincides with Kolmogorov-Sinai entropy (we will not need this fact). However if $\Ga$ is nonamenable the situation is very different:

\begin{thm}
If $\Ga$ is nonamenable then every pmp action of $\Ga$ has naive entropy in $\{0,+\infty\}$.
\end{thm}

\begin{proof}
Suppose $\Ga \cc (X,\mu)$ and there is a finite-entropy partition $\cP$ of $X$ with $h^{naive}_\mu(\cP)>0$. Let $W \subset \Ga$ be finite. Then
$$h^{naive}_\mu(\cP^W) = \inf_{F \subset \subset \Ga} |F|^{-1} H_\mu(\cP^{WF}) = \inf_{F \subset \subset \Ga}  \frac{H_\mu(\cP^{WF})}{|WF|} \frac{|WF|}{|F|} \ge h^{naive}_\mu(\cP) \inf_{F \subset \subset \Ga}  \frac{|WF|}{|F|}.$$
Since $\Ga$ is nonamenable, for every real number $r>0$ there is a finite $W \subset \Ga$ such that
$$\inf_{F \subset \subset \Ga} \frac{|WF|}{|F|}>r.$$
Hence $\sup_{W \subset \subset \Ga} h_\mu^{naive}(\cP^W) = +\infty$ proving the theorem.
\end{proof}


\begin{defn}
Let $\bfa_i = \Ga \cc (X_i,\mu_i)$ be pmp actions (for $i\in I$ where $I$ is some index set). We always assume $I$ is at most countable. A {\bf joining} of these actions is a $\Ga$-invariant Borel probability measure on the produce space $\prod_i X_i$ whose $i$-th marginal is $\mu_i$. Here $\Ga$ acts on the product diagonally: $(\ga x)_i = \ga x_i$. We also refer to the action $\Ga \cc (\prod_i X_i, \lambda)$ as a {\bf joining}. The joining is said to be {\bf finite} if $I$ is finite and {\bf infinite} otherwise. In the special case that $\bfa_i = \bfa_j$ for all $i,j$, the joining is called a {\bf self-joining}. 
\end{defn}

The main result here is:

\begin{prop}\label{prop:closure}
Zero naive entropy is closed under factors, self-joinings (both finite and infinite) and inverse limits.
\end{prop}

We will need the following lemma showing that naive entropy is Lipschitz in the space of partitions.
\begin{lem}\label{lem:lipschitz}
Let $\bfa=\Ga \cc (X,\mu)$ be a pmp action and $\cP,\cQ$ be measurable partitions of $X$ with finite Shannon entropy. Then for any finite $F \subset \Ga$,
$$H_\mu(\cP^F) - H_\mu(\cQ^F) \le |F| H_\mu(\cP|\cQ).$$
Thus
$$h^{naive}_\mu(\cP) - h^{naive}_\mu(\cQ) \le H_\mu(\cP|\cQ).$$
\end{lem}

\begin{proof}
Recall that
\begin{eqnarray*}
H_\mu(\cP^F|\cQ^F) &=& H_\mu(\cP^F \vee \cQ^F) - H_\mu(\cQ^F) \\ 
H_\mu(\cQ^F|\cP^F) &=& H_\mu(\cP^F \vee \cQ^F) - H_\mu(\cP^F).
\end{eqnarray*}
Subtracting, we obtain
\begin{eqnarray*}
H_\mu(\cP^F) - H_\mu(\cQ^F) &=& H_\mu(\cP^F|\cQ^F) - H_\mu(\cQ^F|\cP^F) \le H_\mu(\cP^F|\cQ^F)\\
&\le& \sum_{f\in F} H_\mu(f^{-1}\cP|\cQ^F)  \le |F| H_\mu(\cP|\cQ).
\end{eqnarray*}
This proves the first inequality. The second one follows from the first (observe that we need only consider a sequece of $F$'s that realize the naive entropy for $\cQ$).
\end{proof}

\begin{proof}[Proof of Proposition \ref{prop:closure}]
Let us suppose that $\bfa = \Ga \cc (X,\mu)$ is an inverse limit of actions $\bfa_i = \Ga \cc (X_i,\mu_i)$ having zero naive entropy. We will show $\bfa$ has zero naive entropy. Let $\cF_i$ be the Borel sigma-algebra of $X_i$. After pulling back under the factor map, we may identify $\cF_i$ as a sub-sigma-algebra of the Borel sub-sigma-algebra of $X$ which is denoted here by $\cB_X$. Thus $\cF_1 \subset \cF_2 \subset \cdots$ is an increasing sequence of $\Gamma$-invariant sigma-algebras and $\bigvee_i \cF_i = \cB_X$. Because each action $\bfa_i$ has zero naive entropy, if $\cP$ is any partition of $X$ satisfying $\cP \subset \cF_i$ for some $i$ and $H_\mu(\cP)<\infty$ then necessarily $h^{naive}_\mu(\cP)=0$. 

Let $\cP$ be an arbitrary measurable partition of $X$ with finite Shannon entropy. Since $\inf_i H_\mu(\cP|\cF_i) = 0$, for any $\epsilon>0$ there exists an $i$ and a partition $\cQ \subset \cF_i$ with finite Shannon entropy such that $H_\mu(\cP|\cQ) < \epsilon$. By Lemma \ref{lem:lipschitz}, $h^{naive}_\mu(\cP) \le \epsilon + h^{naive}_\mu(\cQ) = \epsilon$. Since $\epsilon$ and $\cP$  are arbitrary, this implies $\bfa$ has zero naive entropy and therefore zero naive entropy is closed under inverse limits.

Next suppose $\bfa = \Ga \cc (X,\mu)$ has zero naive entropy and let $\lambda$ be a self-joining of $\bfa$. We regard $\lambda$ as a measure on $X\times X$. If $\cP$ is any partition of $X \times X$ with $H_\lambda(\cP)<\infty$ and $\epsilon>0$ is arbitrary then there exists a partition $\cQ$ of $X$ with finite Shannon entropy such that $H_\lambda(\cP | \cQ \times \cQ) <\epsilon$. So Lemma \ref{lem:lipschitz} implies 
$$h^{naive}_\lambda(\cP) \le \epsilon+ h^{naive}_{\lambda}(\cQ \times \cQ).$$
Since $\cQ \times \cQ = (\cQ \times \{X\}) \vee (\{X\} \times \cQ)$, 
\begin{eqnarray*}
h^{naive}_{\lambda}(\cQ \times \cQ) &=& \inf_{F \subset \subset \Ga} \frac{H_\lambda( (\cQ \times \cQ)^F)}{|F|} =  \inf_{F \subset \subset \Ga} \frac{H_\lambda( \cQ^F \times \cQ^F)}{|F|} \\
&\le& \inf_{F \subset \subset \Ga} \frac{H_\lambda( \cQ^F \times \{X\}) + H_\lambda(\{X\} \times \cQ^F)}{|F|}  = \inf_{F \subset \subset \Ga} \frac{2 H_\mu( \cQ^F) }{|F|} =  2 h^{naive}_\mu(\cQ)  = 0.
\end{eqnarray*}
 Thus $h^{naive}_\lambda(\cP)  \le \epsilon$. Since $\epsilon$ and $\cP$  are arbitrary, this implies $\lambda$ has zero naive entropy and by induction, zero naive entropy is closed under finite self-joinings. Any infinite self-joining is an inverse limit of finite self-joinings.  So the above results show that zero naive entropy is closed under infinite self-joinings.  It is immediate from the definitions that zero naive entropy is closed under factors. 

\end{proof}

I do not know whether zero naive entropy is closed under joinings. For example if two actions have zero naive entropy does their direct product also have zero naive entropy?

\section{Five strengthenings of zero entropy}\label{sec:strengthening} 

Here we introduce five strengthenings of the notion of zero entropy. First we need the following definitions: 

\begin{defn}
An action $\Ga \cc (X,\mu)$ has {\bf completely positive Rokhlin entropy} (denoted R-CPE) if every nontrivial factor has positive Rokhlin entropy.
\end{defn}

\begin{defn}
Two actions are said to be {\bf disjoint} if the only joining between them is the product joining. 
\end{defn}

\begin{thm}
Let $\bfa=\Ga \cc (X,\mu)$ be an ergodic essentially free pmp action. Consider the following five properties:
\begin{enumerate}
\item $\bfa$ has {\bf completely zero entropy} (this means every essentially free factor of $\bfa$ has zero Rokhlin entropy);
\item $\bfa$ is disjoint from all Bernoulli shifts over $\Ga$;
\item $\bfa$ is disjoint from all R-CPE actions of $\Ga$;
\item every factor of every self-joining (including infinite self-joinings) of $\bfa$ has zero Rokhlin entropy;
\item $\bfa$ has zero naive entropy.
\end{enumerate}
Then $1 \Leftarrow 2$ and  $3 \Leftarrow 4 \Leftarrow 5$. Moreover, if $\Ga$ is sofic then $2 \Leftarrow 3$. 
\end{thm}

\begin{remark}
When $\Ga$ is amenable, all five properties listed above are equivalent because naive entropy and Rokhlin entropy agree with Kolmogorov-Sinai entropy (at least for ergodic essentially free actions). However when $\Ga$ is nonamenable, it is an open problem whether any or all of the implications above can be reversed. 
\end{remark}

\begin{remark}
If $\Ga$ is nonsofic then we do not know whether Bernoulli shifts over $\Ga$ have positive Rokhlin entropy. This is why we cannot say whether $2 \Leftarrow 3$ unconditionally. See \cite{seward-kreiger-2} for partial results on this problem.
\end{remark}

\begin{proof}
($1 \Leftarrow 2$) This is immediate from Seward's generalization of Sinai's Factor Theorem \ref{thm:seward-sinai}  which states that any ergodic essentially free action with positive entropy surjects onto a Bernoulli shift. Thus if $\bfa$ has a factor with positive entropy then it has a Bernoulli factor $\phi:X \to Y$. The corresponding factor joining is the measure $(\textrm{id}_X \times \phi)_*\mu$. This is a non-product joining. 

($2 \Leftarrow 3$, assuming $\Ga$ is sofic) Since $\Ga$ is sofic, Bernoulli shifts have completely positive entropy by \cite{kerr-cpe}. This uses the fact that sofic entropy is a lower bound for Rokhlin entropy.

($3 \Leftarrow 4$) Let $\bfb$ be another pmp action of $\Ga$ and suppose that $\bfb$ and $\bfa$ admit a nonproduct joining. It follows from the relative independence theorem \cite[Theorem 6.25]{glasner-joinings-book} that there exists an infinite self-joining $\lambda$ of $\bfa$ such that $\Ga \cc (X^\N,\lambda)$ and $\bfb$ admit a nontrivial common factor. Therefore, $\bfb$ cannot be R-CPE.

($4 \Leftarrow 5$) This follows from Proposition \ref{prop:closure} and \cite[Theorem 1.5]{seward-weak-containment} which states that the naive entropy of a generating partition is an upper bound for the Rokhlin entropy. Therefore zero naive entropy implies zero Rokhlin entropy.
\end{proof}

\section{Zero naive entropy}\label{sec:zne}

For an arbitrary group $\Ga$, it is an open problem whether $\Ga$ has an essentially free pmp action with zero naive entropy. However for special classes  of groups we will show not only do such actions exist, they are generic. First we need a definition:

\begin{defn}
The {\bf profinite completion} of $\Ga$ is the inverse limit of the groups of the form $\Ga/N$ where $N\vartriangleleft \Ga$ has finite index in $\Ga$. It is a compact group on which $\Ga$ acts by left translations. The group $\Ga$ is said to be {\bf residually finite} if any one of the following equivalent conditions hold:
\begin{itemize}
\item the action of $\Ga$ on its profinite completion is essentially free;
\item for every non-identity element $g \in \Ga$ there exists a finite-index subgroup $H \le \Ga$ such that $g \notin H$;
\item there exists a decreasing sequence of finite-index normal subgroups $\Ga \ge H_1 \ge H_2 \ge \cdots$ such that $\cap_i H_i = \{1\}$. 
\end{itemize}
\end{defn}

\begin{defn}
Let $p_\Ga$ denote the action of $\Ga$ on its profinite completion by left-translations. This is a pmp action where the measure on the profinite completion is its Haar measure. Also let $\iota$ denote the trivial action of $\Ga$ on the unit interval with respect to Lebesgue measure (the trivial action is the action in which every group element fixes every point). 

A group $\Ga$ has {\bf MD} if the measure conjugacy class of the direct product action $p_\Ga \times \iota$ is dense in the space of actions $A(\Ga,X,\mu)$. Equivalently, $\Ga$ has MD if the subset of measures in $\Prob_\Ga(\Cantor^\Ga)$ with finite support is dense in the weak* topology. This definition is due to Kechris \cite{kechris-2012}; it is a strengthening of property FD which was considered earlier by Lubotzky-Shalom \cite{lubotzky-shalom-2004} in their study of unitary representations.
\end{defn}



\begin{thm}\label{thm:md10}
Free groups, surface groups and fundamental groups of closed hyperbolic 3-manifolds have MD.
\end{thm}

\begin{proof}
The case of free groups was proven independently by Kechris \cite{kechris-2012} and Bowen \cite{MR2026846}. The rest was proven in \cite{bowen-tucker-2013}. The case of fundamental groups of closed hyperbolic 3-manifolds relies on Agol's virtual fibering Theorem \cite{agol-virtual-haken}.
\end{proof}

Let {\bf ZNE} denote the subset of measures $\mu \in  \Prob_\Ga(\Ca^\Ga)$ with  zero naive entropy.

\begin{lem}\label{lem:thing2}
For any countable group $\Ga$, {\bf ZNE} is a $G_\delta$ subset of $\Prob_\Ga(\Ca^\Ga)$. 
\end{lem}

\begin{proof}
Let $\cP_n$ be an increasing sequence of finite clopen partitions of $\Ca^\Ga$ such that $\bigvee_n \cP_n$ is the Borel sigma-algebra. Recall that clopen means every part of $\cP_n$ is both closed and open. Let $A_n$ be the subset of all measures $\mu \in \Prob_\Ga(\Ca^\Ga)$ such that $h^{naive}_\mu(\cP_n)=0$. We claim that $\cap_n A_n = {\bf ZNE}$. Clearly, $\cap_n A_n \supset {\bf ZNE}$. Suppose $\mu \in \cap_n A_n$. Let $\cQ$ be an arbitrary partition of $\Ca^\Ga$ with $H_\mu(\cQ)<\infty$. Then for every $\eps > 0$ there exists $n$ such that $H_\mu(\cQ|\cP_n)<\epsilon$. By Lemma \ref{lem:lipschitz}, $h^{naive}_\mu(\cQ) \le \eps + h^{naive}_\mu(\cP_n) = \eps$. Since $\eps$ and $\cQ$ are arbitrary this proves $\mu \in {\bf ZNE}$ and therefore, $\cap_n A_n = {\bf ZNE}$ as claimed.

It now suffices to show each $A_n$ is a $G_\delta$ subset. Indeed this follows from the definition
$$h^{naive}_\mu(\cP_n) = \inf_{F \subset \subset \Ga} |F|^{-1} H_\mu(\cP_n^F)$$
and the fact that $\mu \mapsto H_\mu(\cP_n^F)$ is weak* continuous for every finite $F \subset \Ga$. The reason this is weak* continuous uses the fact that if $P \subset \Ca^\Ga$ is clopen then its characteristic function is continuous and therefore induces a continuous functional on $\Prob_\Ga(\Ca^\Ga)$.
\end{proof}

\begin{defn}
The {\bf kernel} of an action $\bfa = \Ga \cc (X,\mu)$ is the subgroup $\Ker(\bfa) : = \{g\in \Ga:~gx=x ~\textrm{for}~\textrm{a.e.} ~x\in X\}$. 
\end{defn}

\begin{lem}\label{lem:thing3}
If $\bfa$ has infinite kernel then it has zero naive entropy.
\end{lem}

\begin{proof}
Let $\cP$ be an arbitrary partition of $X$ with finite Shannon entropy. Then $\cP^K=\cP$ (up to measure zero) for every $K \subset \ker(\bfa)$. Therefore,
$$h^{naive}_\mu(\cP) \le \inf_{F \subset \subset \Ker(\bfa)} |F|^{-1} H_\mu(\cP^F) = |\Ker(\bfa)|^{-1} H_\mu(\cP).$$
In paricular if $\Ker(\bfa)$ is infinite then $h^{naive}_\mu(\cP)=0$.
\end{proof}

\begin{proof}[Proof of Theorem \ref{thm:md2}]

By the Glasner-King correspondence mentioned in the introduction, it suffices to show that {\bf ZNE} is a dense $G_\delta$ subset of $\Prob_\Ga(\Ca^\Ga)$. By Lemma \ref{lem:thing2} it is a $G_\delta$.  If $\Ga$ has property MD then, by definition, the subset of all measures $\mu \in \Prob_\Ga(\Ca^\Ga)$ with finite support is dense in $\Prob_\Ga(\Ca^\Ga)$. Each such measure has infinite kernel. So Lemma \ref{lem:thing3} implies {\bf ZNE} is dense. So we assume $\Ga=G \times H$ where $H$ is infinite, amenable and residually finite. 
 
 Because $H$ is residually finite there exists a sequence $H \ge H_1 \ge H_2 \ge \cdots$ of normal finite-index subgroups of $H$ with $\cap_i H_i = \{1_H\}$. By \cite[Theorem 1]{weiss-monotileable}, because $H$ is amenable, there exist right fundamental domains $F_i$ for $H_i$ such that $\{F_i\}$ forms a F\o lner sequence. This means: (1) $H$ is the disjoint union of $H_i f$ over $f \in F_i$ and (2) for any finite $K \subset H$, 
$$\lim_{i\to\infty} \frac{ |\{f\in F_i:~ fK \subset F_i|}{|F_i|} = 1.$$

Let $\mu \in \Prob_\Ga(\Ca^\Ga)$ be arbitrary. We will show that it is a weak* limit of measures with zero Rokhlin entropy. For $i\in \N$, define $\phi_i:\Ca^\Ga \to \Ca^\Ga$ by $\phi_i(x)(g,h) = x(g,f)$ where $g\in G, h \in H$ and $f\in F_i$ is the unique element satisfying $H_ih=H_if$. Observe that $\phi_i(x)$ is $H_i$-invariant and $\phi_i$ is $G$-equivariant. Therefore, the pushforward measure $\phi_{i*}\mu$ is $G\times H_i$-invariant. Also observe that $F_i^{-1}$ is a left fundamental domain in the sense that $H$ is the disjoint union of $f^{-1}H_i$ over $f\in F_i$. Therefore, 
$$\mu_i : = |F_i|^{-1} \sum_{f\in F_i} (1_G,f_i^{-1})_*\phi_{i*}\mu$$
is $\Ga$-invariant. Since $H_i$ is normal, the kernel of the action $\Ga \cc (\Ca^\Ga,\mu_i)$ contains $H_i$. By Lemma \ref{lem:thing3}, this action has zero naive entropy. 

We claim that $\mu_i \to \mu$ as $i\to\infty$. To see this, let $\Phi_i:\Ca^\Ga \to \Ca^\Ga \times \Ca^\Ga$ denote the graph of $\phi_i$:
$$\Phi_i(x) = (x, \phi_i(x)).$$
Let $\lambda_i =|F_i|^{-1}\sum_{f\in F_i} (1_G,f_i^{-1})_*\Phi_{i*}\mu$. Because $\lambda_i$ is a joining of $\mu$ and $\mu_i$ it suffices to show that for every $(g,h) \in G\times H$,
$$\lambda_i(\{ (x,y):~ x(g,h)=y(g,h)\} ) \to 1$$
as $i\to\infty$. So fix $(g_0,h_0) \in G\times H$. To simplify notation, we let 
$$\Delta = \{(x,y) \in \Ca^\Ga \times \Ca^\Ga:~x(g_0,h_0) = y(g_0,h_0)\}.$$
 It suffices to show that for any $x\in \Ca^\Ga$,
$$\frac{\#\{f\in F_i:~ (1_G,f_i^{-1})\Phi_i(x) \in \Delta\}}{\#F_i}  \ge  \frac{ |\{f\in F_i:~ fh_0 \in F_i \} |}{|F_i|}$$
since the latter tends to 1 uniformly in $x$. This follows from
$$\{f\in F_i:~ (1_G,f_i^{-1})\Phi_i(x) \in \Delta \} \supset \{f\in F_i:~ fh_0 \in F_i \}$$
which follows directly from the definitions: if $f\in F_i$ and $fh_0 \in F_i$ then 
$$(1_G,f^{-1})\Phi_i(x)(g_0,h_0) = \Phi_i(x)(g_0,fh_0) = (x(g_0,fh_0), x(g_0, fh_0)).$$
This proves the claim. This implies that $\mu_i \to \mu$ as $i\to\infty$ in the weak* topology. Indeed, if $L \subset \Ga$ is any finite subset and $f:\Ca^L \to \C$ any continuous function then the function $\tf:\Ca^\Ga \to \C$ defined by composing the restriction map $\Ca^\Ga \to \Ca^L$ with $f$ satisfies $\int \tf~d\mu_i \to \int f~d\mu$. Since such functions are dense in the space of all continuous functions, it follows that $\mu_i \to \mu$ as claimed. Because $\mu$ is arbitrary, this implies {\bf ZNE} is dense.

\end{proof}

\section{Weak containment}\label{sec:weak}

Given any pmp action $\bfa=\Ga \cc (X,\mu)$, let $\Factor(\bfa)$ denote the set of all measures $\nu \in \Prob_\Ga(\Cantor^\Ga)$ such that there is a $\Ga$-equivariant measurable map $\phi:X \to \Cantor^\Ga$ with $\phi_*\mu=\nu$. This is the set of {\bf factor measures}. Let $W(\bfa)$ be the weak* closure of $\Factor(\bfa)$. 

Now suppose $\bfb=\Ga \cc (Y,\nu)$ is another pmp action. We say $\bfb$ is {\bf weakly contained} in $\bfa$, denoted $\bfb \prec \bfa$, if $W(\bfb) \subset W(\bfa)$. If $\bfb \prec \bfa$ and $\bfa \prec \bfb$ then we say $\bfa$ and $\bfb$ are {\bf weakly equivalent}. This notion was introduced in \cite{kechris-2012}. In \cite{T-D12} it is proven that the definition given in this paper is equivalent to the one introduced in \cite{kechris-2012}.  Some basic facts: all Bernoulli shifts over $\Ga$ are weakly equivalent. In fact the Abert-Weiss Theorem \cite{abert-weiss-2013} states: if $\bfa$ is any essentially free action of $\Ga$ then $\bfa$ weakly contains a Bernoulli shift. There exists an action $\bfa$ that weakly contains all actions of $\Ga$ (this is called the weak Rokhlin property, see \cite{glasner2006every}). 

It is an open problem whether, for a given action $\bfa$, the set of all measures $\mu \in W(\bfa)$ with zero Rokhlin entropy is residual. Of course, this is true if $W(\bfa) = \Prob_\Ga(\Ca^\Ga)$ by Theorem \ref{thm:zero}. It is also true if $\bfa$ is a Bernoulli shift:

\begin{cor}\label{cor:weak-zero}
Let $\bfa$ be a Bernoulli shift. Then the generic measure $\mu \in W(\bfa)$ has zero Rokhlin entropy.
\end{cor}

\begin{proof}
If $\Ga$ is amenable then $W(\bfa)=\Prob_\Ga(\Ca^\Ga)$. So the result follows from Theorem \ref{thm:zero}. So we may assume $\Ga$ is nonamenable. In this case, $\bfa$ is strongly ergodic and therefore every measure $\mu \in W(\bfa)$ is ergodic. By Lemma \ref{lem:g-delta1}, the set of all measures $\mu \in W(\bfa)$ with zero Rokhlin entropy is a $G_\delta$ subset. By Corollary \ref{cor:inverse limit} there exists an action $\bfb$ that is an inverse limit of factors of Bernoulli shifts, that factors onto all Bernoulli shifts and has zero Rokhlin entropy. By Lemma \ref{lem:2}, $W(\bfb)=W(\bfa)$. By Lemma \ref{lem:2} again, the set of measures in $W(\bfb)$ with zero Rokhlin entropy is dense.
\end{proof}

\appendix

\bibliography{biblio}
\bibliographystyle{alpha}

\end{document}